\documentclass[12pt]{amsart}

\usepackage[T1]{fontenc}
\usepackage[latin1]{inputenc}
\usepackage[english]{babel}
\usepackage{amsmath,amssymb,amsthm} 
\usepackage{enumitem}
\usepackage{mymacros}
\usepackage[margin=1.25in]{geometry}
\usepackage{hyperref}

\usepackage{color}

\newcommand\black{\color{black}}

\author[K. Bickel]{Kelly Bickel}
\address{Department of Mathematics, Bucknell University, 701 Moore Ave, Lewisburg, PA 17837, USA}
\email{kelly.bickel@bucknell.edu}
\thanks{K.B. was partially supported by National Science Foundation Grant
DMS 1448846}

\author[M. Hartz]{Michael Hartz}
\address{Department of Mathematics, Washington University in St. Louis, One Brookings Drive,
St. Louis, MO 63130, USA}
\email{mphartz@wustl.edu}
\thanks{M.H. was partially supported by a Feodor Lynen Fellowship}

\author[J. M\raise.5ex\hbox{c}Carthy]{John E. M\raise.5ex\hbox{c}Carthy}
\address{Department of Mathematics, Washington University in St. Louis, One Brookings Drive,
St. Louis, MO 63130, USA}
\email{mccarthy@wustl.edu}
\thanks{J.M. was partially supported by National Science Foundation Grant
DMS 1565243}


\renewcommand{\MR}[1]{}

\title{A multiplier algebra functional calculus}

\begin{document}

\begin{abstract}
This paper generalizes the classical Sz.-Nagy--Foias $H^{\infty}(\mathbb{D})$ functional calculus for Hilbert space contractions. In particular, we replace the single contraction $T$ with a tuple $T=(T_1, \dots, T_d)$ of commuting bounded operators on a Hilbert space and replace $H^{\infty}(\mathbb{D})$ with a large class of multiplier algebras
of Hilbert function spaces on the unit ball in $\bC^d$.

\end{abstract}

\maketitle

\section{Introduction}
\subsection{Overview} One seminal result connecting operator theory and complex function theory is the classical Sz.-Nagy--Foias $H^{\infty}(\mathbb{D})$ functional calculus \cite{Nagy10}, which  says that for a completely non-unitary contraction $T$ on a Hilbert space $\mathcal{K}$, the polynomial functional calculus
\begin{equation} \label{eqn:poly} \Phi_T: \mathbb{C}[z] \rightarrow B(\mathcal{K}), \quad T \mapsto p(T), \end{equation}
extends to a weak-$*$ continuous, contractive algebra homomorphism on $H^{\infty}(\mathbb{D})$ (with related statements for general contractions).  This allows one to make sense of $f(T)$ for fairly general functions $f$ and operators $T$ and has both illuminated the structure of $C_0$ contractions and related operators, see \cite{BER02, BC88, FSN64, Nagy10}, and  led to breakthroughs related to the invariant subspace problem, see \cite{BCP79, BCP88}.

In this paper, we generalize the Sz.-Nagy--Foias $H^{\infty}(\mathbb{D})$ functional calculus in two nontrivial ways:~first, we replace the single contraction $T \in B(\mathcal{K})$ with a tuple $T=(T_1, \dots, T_d)$ of commuting (but not necessarily contractive) bounded operators on $\mathcal{K}$ and secondly, we replace $H^{\infty}(\mathbb{D})$ with a multiplier algebra associated to any of a large class of reproducing kernel Hilbert spaces $\mathcal{H}$ on the open unit ball $\mathbb{B}_d \subseteq \mathbb{C}^d$.

Results of this type were proved by J.~Eschmeier, who established an $H^{\infty}(\mathbb{B}_d)$-functional calculus for completely non-unitary tuples of commuting operators satisfying von Neumann's inequality over the unit ball \cite{Eschmeier97} and by R.~Clou\^atre and K.~Davidson, who studied the Drury-Arveson space $H^2_d$  on $\mathbb{B}_d$ and its multiplier algebra $\Mult(H^2_d)$. They proved that completely non-unitary commuting row contractions admit a $\Mult(H^2_d)$-functional calculus  \cite{CD16a}. This current paper generalizes these earlier investigations and provides new arguments that avoid technical analyses of specific multiplier algebras.

\subsection{Main Result}  Here is our setting: Let $T=(T_1, \dots, T_d)$ be a tuple of commuting operators on a Hilbert space $\mathcal{K}$ and let $U = (U_1,\ldots,U_d)$ be a \emph{spherical unitary} on $\cK$, i.e.~ each $U_i$ is normal and $\sum U_i U_i^* = I$. We say $T$ is \emph{completely non-unitary} if it has no non-zero reducing subspace $\cM$ with $T|_{\cM}$ a spherical unitary.

Let $\mathcal{H}$ be a reproducing kernel Hilbert space on $\mathbb{B}_d$ whose multiplier algebra, $\Mult(\mathcal{H})$, contains the polynomials $\mathbb{C}[z_1, \dots, z_d]$. Thus, the polynomial functional calculus
for $T$,
\begin{equation} \label{eqn:poly2} \Phi_T: \mathbb{C}[z_1, \dots, z_d] \rightarrow B(\mathcal{K}), \quad
  T \mapsto p(T),\end{equation}
is defined on a subalgebra of $\Mult(\mathcal{H})$. 

We are interested in well-behaved extensions of the polynomial functional calculus $\Phi_T$, namely ways to make sense of $f(T)$ for a larger class of functions. Let $A(\mathcal{H})$ denote the norm closure of the polynomials in $\Mult(\mathcal{H})$. Then,
 we say \emph{$T$ admits an $A(\cH)$-functional calculus} if $\Phi_T$ extends to a completely contractive algebra homomorphism $A(\mathcal{H}) \rightarrow B(\mathcal{K})$;
 this loosely says $T$ admits an
  $A(\cH)$ 
  version of von Neumann's inequality. Moreover, we say \emph{$T$ is $\Mult(\cH)$-absolutely continuous} if in addition $\Phi_T$ extends to a weak-$*$ continuous algebra homomorphism
 $\Mult(\mathcal{H}) \rightarrow B(\mathcal{K})$.
Thus, the following two questions naturally arise:
\begin{itemize}
\item[(Q1)] When does $T$ admit an $A(\cH)$-functional calculus?
\item[(Q2)] If $T$ admits an $A(\cH)$-functional calculus, is $T$ $\Mult(\cH)$-absolutely continuous?
\end{itemize}

If $\cH$ is a complete Nevanlinna-Pick space, then
(Q1) has a concrete answer in terms of the reproducing kernel of $\cH$ (see Section \ref{sec:cnp}).
However, even in the case when $\cH$ is the Hardy space on the unit ball, so that $A(\cH)$
is the ball algebra, we are not aware of an explicit answer to (Q1). 
Nevertheless,
Arveson's dilation theorem shows that one can rephrase (Q1) in terms of dilations,
and the resulting dilation has a very explicit form if $\cH$
satisfies some additional assumptions, which we now describe.
We say that $\cH$
is a \emph{unitarily invariant space} on $\bB_d$ if the reproducing kernel $K$ of $\cH$ is of the form
\begin{equation*}
  K(z,w) = \sum_{n=0}^\infty a_n \langle z,w \rangle^n \quad \text{ with } a_0 = 1
  \text{ and } a_n > 0 \text{ for all } n \in \bN.
\end{equation*}
If in addition $\lim_{n \to \infty} a_n / a_{n+1} = 1$,
we say that $\cH$ is a \emph{regular unitarily invariant space}.
A discussion of this condition can be found in Subsection \ref{ss:rkhs}. For now,
we merely mention that this class of spaces includes many well-known
spaces on $\bB_d$ such as the Bergman space, the Hardy space,
the Dirichlet space and the Drury-Arveson space.

Then our main result is the following answer to (Q2) for completely non-unitary tuples $T$:

\begin{thm}
  \label{thm:cnu_abs_con}
  Let $\cH$ be a regular unitarily invariant space on $\bB_d$. Let
  $T = (T_1,\ldots,T_d)$ be a tuple of commuting operators on a Hilbert space that
  admits an $A(\cH)$-functional calculus.
  If $T$ is completely non-unitary, then $T$ is $\Mult(\cH)$-absolutely continuous.
\end{thm}

This loosely says that if $T$ satisfies 
an $A(\cH)$
\black
  version of von Neumann's inequality and
is completely non-unitary, then we can make sense of $f(T)$ for all $f \in \text{Mult}(\mathcal{H})$.
This theorem partially extends Corollary 1.7 in \cite{Eschmeier97},
which establishes an $H^\infty(\bB_d)$-functional calculus for completely non-unitary tuples which satisfy
von Neumann's inequality over the unit ball.
Moreover, if $\cH = H^2_d$, the Drury-Arveson space, then a theorem of M\"uller--Vasilescu \cite{MV93}
and Arveson \cite{Arveson98} shows that a tuple of commuting operators $T$
admits an $A(\cH)$-functional calculus if and only if $T$ is a row contraction, i.e. $\sum_{i=1}^d T_i T_i^* \le I$.
Thus, in this setting, we recover Theorem 4.3 in \cite{CD16a}. We remark that in this last case, our proof does
not require the detailed description of the dual of $A(H^2_d)$ from \cite{CD16},
which was used in \cite{CD16a}.

Even in the one variable setting, Theorem \ref{thm:cnu_abs_con} appears to be new. For example, consider the one-variable Dirichlet space $\mathcal{D}$. Because the Dirichlet space has an irreducible complete Nevanlinna-Pick kernel, there is an explicit characterization (see Theorem \ref{thm:H-coextension}) for when $T$ admits an $A(\cD)$-functional calculus, and this characterization does not require $T$ to be a contraction.
Then for such $T$ which are completely non-unitary,
Theorem \ref{thm:cnu_abs_con} says that  $T$ is $\Mult(\cD)$-absolutely continuous.


\subsection{Outline of Paper}
Recall that the Sz.-Nagy--Foias $H^{\infty}(\mathbb{D})$ functional 
calculus requires two structural results about contractions: first, each contraction
$T$ decomposes as $T = T_{cnu} \oplus U$, with $T_{cnu}$ 
 completely non-unitary and $U$ unitary; and second, Sz.-Nagy's dilation theorem, namely
 each contraction admits a unitary dilation.
If $T=U$ is unitary, the spectral theorem implies that $U$ has a weak-$*$ continuous
$H^{\infty}(\mathbb{D})$ functional calculus
 if and only if its scalar spectral measure is absolutely continuous with respect to Lebesgue measure.
 If $T=T_{cnu}$ is completely non-unitary, then the spectral measure of its minimal unitary
 dilation is absolutely continuous with respect to Lebesgue measure, from which
 the general calculus follows \cite{Nagy10}.

This paper follows the classical outline. In Section \ref{sec:structure}, we recall
some necessary background material on reproducing kernel Hilbert spaces and multivariable
operator theory.
In particular, Proposition \ref{prop:cnu_decom} notes that every tuple of commuting operators $T$ 
splits as $T= T_{cnu} \oplus U$, for  $T_{cnu}$ completely non-unitary and $U$ a 
spherical unitary. Using standard dilation theoretic arguments, we show in Theorem \ref{thm:H-dilation}
that (Q1) can be rephrased in terms of concrete dilations of $T$ for regular
unitarily invariant spaces $\cH$.

Section \ref{sec:henkin} develops some properties of $\Mult(\cH)$-Henkin  
measures, the analogue of measures which are absolutely continuous
with respect to Lebesgue measure in our setting. This section closely
follows part of the work of Clou\^atre and Davidson on $\cA_d$-Henkin measures \cite{CD16}.
In particular, adapting their arguments, we extend in Lemma \ref{lem:band} a theorem of Henkin by showing
that $\Mult(\cH)$-Henkin measures form a band.
Moreover, Lemma  \ref{lem:unitary_abs_cont} gives
the following answer to (Q2) for spherical unitaries: a
spherical unitary $U$ is $\Mult(\cH)$-absolutely continuous if and 
only if its scalar spectral measure is $\Mult(\cH)$-Henkin. In Proposition \ref{prop:ac_dilation}, we
also provide an answer to (Q2) for operator tuples $T$ in terms of a minimal dilation of $T$.

Section \ref{sec:ac} proves our main result, Theorem \ref{thm:cnu_abs_con}. 
The proof requires the dilation theorem (Theorem \ref{thm:H-dilation}) as well as Lemma \ref{lem:reductive},
which says analytic polynomials are weak-$*$ dense in $L^\infty(\mu)$, for 
$\mu$ a totally singular measure.
Then Theorem \ref{thm:general_functional_calculus} also answers (Q2) for commuting operator tuples
which may have spherical unitary summands.

In Section \ref{sec:cnp}, we consider the case when $\cH$ is a complete Nevanlinna-Pick space.
In this setting, Corollary \ref{cor:NP_functional_calculus} characterizes
those operator tuples for which (Q1) and (Q2) have affirmative answers.

\section{Preliminaries} \label{sec:structure}

\subsection{Unitarily invariant spaces}
\label{ss:rkhs}
We begin by discussing unitarily invariant reproducing kernel Hilbert spaces on the unit ball $\bB_d$; in what follows,
we only consider the situation $d < \infty.$
Background material on reproducing kernel Hilbert spaces can be found in
\cite{AM02} and \cite{PR16}. 
Recall from the introduction that
a \emph{unitarily invariant space} on $\bB_d$ is a reproducing kernel Hilbert space on $\bB_d$
whose reproducing kernel $K$ is of the form
\begin{equation*}
  K(z,w) = \sum_{n=0}^\infty a_n \langle z,w \rangle^n \quad \text{ with } a_0 = 1
  \text{ and } a_n > 0 \text{ for all } n \in \bN.
\end{equation*}
Moreover, we say that $\cH$ is
a \emph{regular unitarily invariant space} if in addition, we have $\lim_{n \to \infty} a_n / a_{n+1} = 1$.
Since $\cH$ is a space on $\bB_d$, it is natural to assume that the
radius of convergence of the power series $\sum_{n=0}^\infty a_n t^n$ is $1$, so
that the limit $\lim_{n \to \infty} a_n / a_{n+1}$, if it exists, is necessarily equal to $1$. Thus,
we regard this condition as a regularity condition on the sequence $(a_n)$.
If $\cH$ is a regular unitarily invariant space, then the polynomials are automatically multipliers
of $\cH$ (see, for example \cite[Corollary 4.4 (1)]{GHX04}).
Regularity is very useful to us since it guarantees that the tuple $(M_{z_1},\ldots,M_{z_d})$
is a spherical unitary tuple modulo the compacts (see the proof of Theorem \ref{thm:H-dilation}
below).

Examples of regular unitarily invariant spaces are the spaces with reproducing kernels
\begin{equation*}
  K(z,w) = \frac{1}{(1- \langle z,w \rangle)^\alpha},
\end{equation*}
where $\alpha \in (0,\infty)$, which includes in particular the Drury-Arveson space ($\alpha = 1$),
the Hardy space ($\alpha = d$) and the Bergman space ($\alpha = d+1$).
A related scale of spaces is given by the kernels
\begin{equation*}
  K(z,w) = \sum_{n=0}^\infty (n+1)^s \langle z,w \rangle^n,
\end{equation*}
where $s \in \bR$.
Here, the Drury-Arveson space corresponds to the choice $s=0$. If $s=-1$, then
we obtain the Dirichlet space on $\bB_d$.

\subsection{The \texorpdfstring{weak-$*$}{weak-*} topology on the multiplier algebra} Let $\cH$ be a unitarily
invariant space on $\bB_d$. Since $1 \in \cH$, every multiplier $\varphi \in \Mult(\cH)$
is uniquely determined by its associated multiplication operator $M_\varphi$. Thus,
the assignment $\varphi \mapsto M_\varphi$ allows us to regard $\Mult(\cH)$ as a unital
subalgebra of $\cB(\cH)$. It is not hard to see that $\Mult(\cH)$
is closed in the weak operator topology and a fortiori in the weak-$*$ topology.
Therefore, $\Mult(\cH)$ is a dual space in its own right, namely the dual
of $\cB(\cH)_* / \Mult(\cH)_\bot$. We endow $\Mult(\cH)$ with
the resulting weak-$*$ topology.

Since unitarily invariant spaces are separable, so is $\cB(\cH)_* / \Mult(\cH)_\bot$,
hence the weak-$*$ topology on bounded subsets of $\Mult(\cH)$ is metrizable.
Moreover, one easily verifies that on bounded subsets of $\Mult(\cH)$, the weak-$*$ topology
coincides with the topology of pointwise convergence on $\bB_d$.

\subsection{Decomposition of operator tuples}  The proof of \cite[Theorem 4.1]{CD16a} implicitly contains the fact that every commuting operator
tuple $T = (T_1, \dots, T_d)$ is the direct sum of a completely non-unitary tuple and a spherical unitary.
The authors of \cite{CD16a} attribute the argument to J\"org Eschmeier. For completeness, we repeat
the relevant part of the argument, beginning with this lemma:

\begin{lem}
  \label{lem:spherical_unitary_char}
  Let $U = (U_1,\ldots,U_d)$ be a tuple of commuting operators on a Hilbert space.
  Then $U$ is a spherical unitary if and only if
  \begin{equation}
    \label{eqn:spherical_unitary}
    \sum_{k=1}^d U_k U_k^* = \sum_{k=1}^d U_k^* U_k = I.
  \end{equation}
\end{lem}

\begin{proof}
  It is obvious that every spherical unitary satisfies \eqref{eqn:spherical_unitary}.
  Conversely, suppose that $U$ satisfies \eqref{eqn:spherical_unitary}.
  Then $U$ and $U^*$ are spherical isometries.
  By a theorem of Athavale \cite{Athavale90}, $U$ therefore extends to a spherical unitary $V = (V_1,\ldots,V_d)$,
  say
  \begin{equation*}
    V_k =
    \begin{bmatrix}
      U_k & A_k \\ 0 & B_k
    \end{bmatrix}.
  \end{equation*}
  From $\sum_{k=1}^d V_k V_k^* = I$ and $\sum_{k=1}^d U_k U_k^* = I$, we deduce
  \begin{equation*}
    \begin{bmatrix}
      I & 0 \\ 0 & I
    \end{bmatrix}
    =
    \begin{bmatrix}
      I + \sum_{k=1}^d A_k A_k^* & \ast \\ \ast & \ast
    \end{bmatrix},
  \end{equation*}
  so that $A_k = 0$ for $k = 1,\ldots,d$. Hence, $U$ is the restriction
  of the spherical unitary $V$ to a reducing subspace, so $U$ itself
  is a spherical unitary.
\end{proof}

The desired decomposition of operator tuples follows easily. 
Note that there are no contractivity assumptions on the operator tuple.

\begin{prop}
  \label{prop:cnu_decom}
  Let $T$ be a tuple
  of commuting operators on a Hilbert space $\cK$. Then
  there exist unique complementary reducing subspaces $\cK_{cnu}$ and $\cK_u$
  for $T$ such that $T_{cnu} :=T \big|_{\cK_{cnu}}$ is completely non-unitary
  and $U:=T \big|_{\cK_u}$ is a spherical unitary.
\end{prop}

\begin{proof}
  Let $\cU$ be the set of all reducing subspaces $M$ of $T$ with the property that
  $T \big|_{M}$ is a spherical unitary and let $\cK_u$ be the closed linear
  span of all elements of $\cU$.
  Then $\cK_u$ is reducing for $T$,
  and if we define $\cK_{cnu} = \cK \ominus \cK_u$, then $T \big|_{\cK_{cnu}}$ is completely
  non-unitary by definition.

  Let $U = T \big|_{\cK_u}$. Observe
  that each $M \in \cU$ is contained in the space
  \begin{equation*}
    \ker \left( I - \sum_{k=1}^d U_k U_k^* \right) \cap \ker \left( I - \sum_{k=1}^d U_k^* U_k \right)
  \end{equation*}
  and hence, so is $\cK_u$. This shows that $\sum_{k=1}^d U_k U_k^* = \sum_{k=1}^d U_k^* U_k = I_{\cK_u}$,
  so that $U$ is a spherical unitary by Lemma \ref{lem:spherical_unitary_char}.

  Finally, if $\cK = \cK_u' \oplus \cK_{cnu}'$ is another decomposition as in the statement,
  then by definition, $\cK_u' \subset \cK_u$. On the reducing
  subspace $\cK_u \ominus \cK_u' = \cK_{cnu} ' \cap \cK_u$, the tuple $T$
  is both spherical unitary and completely non-unitary. Hence, this subspace is trivial,
  so $\cK_u' = \cK_u$ and thus, $\cK_{cnu}' = \cK_{cnu}$.
\end{proof}

\begin{rem}
  According to our definition, an operator tuple $T$ is completely non-unitary
  if and only if $T$ has no non-zero \emph{reducing} subspace $M$ with $T|_{M}$ a spherical unitary.
  If $T$ is a row contraction, then it is not hard to check that if $T$ is completely
  non-unitary in our sense,
  then it does not even have a non-zero \emph{invariant} subspace $M$ with $T|_{M}$ a spherical unitary.
  This is the definition used in \cite{CD16a}.
\end{rem}

\subsection{Dilations} \label{ss:dilations}
If $T = (T_1,\ldots,T_d)$ is a tuple of commuting operators on a Hilbert space $\cK$, and
if $S = (S_1,\ldots,S_d)$ is a tuple of commuting operators on a larger Hilbert space $\cL \supset \cK$,
then we say that \emph{$T$ dilates to $S$} if $\cL$ decomposes as $\cL = \cL_- \oplus \cK \oplus \cL_+$
such that with respect to this decomposition,
\begin{equation*}
  S_k =
  \begin{bmatrix}
    \ast & 0 & 0 \\
    \ast & T_k & 0 \\
    \ast & \ast & \ast
  \end{bmatrix}, \qquad \text{ for } k=1, \dots, d.
\end{equation*}

Recall that the polynomials
are multipliers of every regular unitarily invariant space $\cH$.
We let $M_z$ denote the $d$-shift on $\mathcal{H}$,
namely the tuple of multiplication operators  $(M_{z_1}, \dots, M_{z_d})$.
Abstract dilation theory in the form of Arveson's dilation theorem yields
the following result (cf. \cite[Section 8]{Arveson98}).

\begin{thm}
  \label{thm:H-dilation}
Let $\cH$ be a regular unitarily invariant space on $\bB_d$ and let $T$ be a tuple of commuting operators on a Hilbert space $\cK$. Then the following
are equivalent:
\begin{enumerate}[label=\normalfont{(\roman*)}]
  \item $T$ admits an $A(\cH)$-functional calculus.
  \item $T$ dilates to $M_z^{\kappa} \oplus U$ for some cardinal $\kappa$
    and spherical unitary $U$.
\end{enumerate}
Moreover, if $\cK$ is separable, then the dilation $M_z^{\kappa} \oplus U$
can be realized on a separable space.
\end{thm}

\begin{proof}
  The implication (ii) $\Rightarrow$ (i) is clear, since $M_z^\kappa$ tautologically
  admits an $A(\cH)$-functional calculus, and $U$ admits an $A(\cH)$-functional calculus
  since the multiplier norm dominates the supremum norm.

  Conversely, suppose that (i) holds. By Arveson's dilation theorem
  (see, for example, \cite[Corollary 7.7]{Paulsen02}), the $A(\cH)$-functional
  calculus for $T$ dilates to a $*$-representation $\pi: C^*(A(\cH)) \to \cB(\cL)$, for some Hilbert space $\cL.$ A result
  of Sarason (see, for example, \cite[Lemma 10.2]{AM02}) implies that $\cL$ decomposes
  as $\cL = \cL_- \oplus \cK \oplus \cL_+$ such that
  \begin{equation*}
    \pi(M_{z_k}) =
  \begin{bmatrix}
    \ast & 0 & 0 \\
    \ast & T_k & 0 \\
    \ast & \ast & \ast
  \end{bmatrix}
  \end{equation*}
  for $k = 1,\ldots,d$. It remains to show that $\pi(M_{z})$ is unitarily equivalent to $M_z^{\kappa} \oplus U$.

  Since $\lim_{n \to \infty} a_n / a_{n+1} = 1$, an application of \cite[Theorem 4.6]{GHX04} shows
  that there exists a short exact sequence of $C^*$-algebras
  \begin{equation*}
    0 \longrightarrow \cK(\cH) \longrightarrow C^*(A(\cH)) \longrightarrow C( \partial \bB_d) \longrightarrow 0,
  \end{equation*}
  where $\cK(\cH)$ is the ideal of all compact operators on $\cH$, the first map is the inclusion map
  and the second map sends $M_{z_k}$ to $z_k$ for $k=1,\ldots,d$. In this setting, general results about
  representations of $C^*$-algebras (see, for example, the discussion preceding Theorem 1.3.4
  and Corollary 2 in Section 1.4 in \cite{Arveson76})
  show that $\pi$ splits as an orthogonal direct sum $\pi = \pi_1 \oplus \pi_2$, where $\pi_1$
  is unitarily equivalent to a multiple of the identity representation  and $\pi_2$ annihilates
  the compact operators and hence can be regarded as a representation of $C(\partial \bB_d)$.
  Consequently, $\pi_1(M_z)$ is unitarily equivalent to $M_z^{\kappa}$ for some cardinal $\kappa$,
  and $\pi_2(M_z)$ is a spherical unitary tuple.

  The additional claim follows because if $\cH$ is separable, then the dilation $\pi$,
  which is obtained from Stinespring's dilation theorem, can be realized on a separable
  space (see the discussion following \cite[Theorem 4.1]{Paulsen02}).
\end{proof}

In the setting of Theorem \ref{thm:H-dilation}, we say that $M_z^{\kappa} \oplus U \in \cB(\cL)^d$ is a
\emph{minimal dilation}
of $T$ if $\cL$ is the smallest reducing subspace for $M_z^\kappa \oplus U$ that contains $\cK$.
If $T$ acts on a separable Hilbert space, then so does every minimal dilation of $T$.
Note that minimial dilations of this type are generally not unique up to unitary equivalence.
Indeed, if $\cH = H^2$, the Hardy space on the unit disc, and if $T = M_z \in \cB(H^2)$,
then $M_z$ itself and the bilateral shift are both minimal dilations of $M_z$ of the form $M_z \oplus U$,
but they are not unitarily equivalent.

\subsection{Normal tuples}
\label{ss:normal}
We briefly recall a few notions from the spectral
theory for commuting normal operators. Let $N = (N_1,\ldots,N_d)$ be a commuting tuple of normal operators
on a Hilbert space $\cK$. By the Putnam-Fuglede theorem, the unital $C^*$-algebra $C^*(N) = C^*(N_1,\ldots,N_d)$
is commutative. Moreover, the maximal ideal space $X$ of $C^*(N)$ can be identified
with a compact subset of $\bC^d$ via
\begin{equation*}
  X \rightarrow \bC^d, \quad \rho \mapsto (\rho(N_1),\ldots,\rho(N_d)).
\end{equation*}
The range of this map is called the \emph{joint spectrum of $N$}, denoted by $\sigma(N)$.
By the
Gelfand-Naimark theorem, $C^*(N)$ is isomorphic to $C(\sigma(N))$ via a $*$-homomorphism
that sends $N_i$ to $z_i$ for each $i=1,\ldots,d$.

A well-known result about representations of algebras of continuous functions (see
\cite[Theorem IX.1.14]{Conway90}) shows that there exists
a \emph{projection-valued spectral measure} $E$ on $\sigma(N)$ such that
\begin{equation*}
  p(N,N^*) = \int_{\sigma(N)} p(z,\ol{z})\, d E
\end{equation*}
for every $p \in \bC[z_1,\ldots,z_d, \ol{z_1},\ldots,\ol{z_d}]$. If $\cK$ is separable, then $W^*(N)$,
the von Neumann algebra generated by $N$, admits a separating vector $x \in \cK$
(see \cite[Corollary IX.7.9]{Conway90}), so
the scalar-valued measure $\mu = \langle E(\cdot)x,x \rangle$ has the same null sets as $E$. Any
scalar-valued measure with this property is called a \emph{(scalar-valued) spectral measure for $N$}.
It is unique up to mutual absolute continuity. Moreover, if $\mu$ is a scalar-valued spectral
measure for $N$, then the isomorphism between $C^*(N)$ and $C(\sigma(N))$ extends
to a $*$-isomorphism between $W^*(N)$ and $L^\infty(\mu)$ which is a weak-$*$-weak-$*$ homeomorphism
(see \cite[Theorem II.2.5]{Davidson96}).

\section{Henkin measures} \label{sec:henkin}

We begin by recalling a few definitions of measures 
on the unit sphere which are related to the ball algebra. Background
material on this topic can be found in \cite[Chapter 9]{Rudin08}.
Let
$M(\partial \mathbb{B}_d)$ denote the space of complex regular Borel measures on $\partial \mathbb{B}_d$, which
can be identified with the 
dual space of $C(\partial \mathbb{B}_d)$, the set of continuous functions on $\partial \mathbb{B}_d$. 
A probability measure $\mu \in M(\partial \mathbb{B}_d)$ is called a \emph{representing
measure for the origin} if
\begin{equation*}
  p(0) = \int_{\partial \bB_d} p \, d \mu
\end{equation*}
holds for every polynomial $p$ (or, equivalently, for every function in the ball algebra).
A Borel set $E$ is said to be \emph{totally null} if it is null for every representing
measure of the origin.
A measure $\mu \in M(\partial \bB_d)$ is said to be \emph{totally singular} if it is singular
with respect to every representing measure of the origin.

Let $\mathcal{H}$ be a unitarily invariant space on $\mathbb{B}_d$ such that
the polynomials are multipliers of $\cH$. Let $A(\cH)$ denote
the norm closure of the polynomials.
Since the supremum norm is dominated by the multiplier norm, 
$A(\cH)$ is contained in the ball algebra and so, every
function in $A(\cH)$ extends uniquely to a continuous function on $\ol{\bB_d}$.

Then, we say that a functional $\rho \in A(\cH)^*$ is \emph{$\Mult(\cH)$-Henkin}
if it extends to a weak-$*$ continuous functional on $\Mult(\cH)$.
We say that a measure $\mu \in M(\partial \mathbb{B}_d)$
is a \emph{$\Mult(\cH)$-Henkin measure} if the functional $\rho_{\mu} \in  A(\cH)^*$ defined by
\begin{equation*}
\rho_{\mu}(f) = \int_{\partial \bB_d} f \, d \mu, \qquad f \in A(\cH)
\end{equation*}
is $\Mult(\cH)$-Henkin.

The classical Henkin measures are precisely the $H^\infty(\bB_d)$-Henkin measures
in our terminology. Moreover, $\Mult(H^2_d)$-Henkin measures
are called $\cA_d$-Henkin in \cite{CD16a}.

The following lemma is a straightforward generalization of \cite[Lemma 1.1]{Eschmeier97}.
In particular, specializing to the case $\cK = \bC$, we see
that the notion of $\Mult(\cH)$-Henkin measures (or functionals)
only depends on $\Mult(\cH)$
and not on the particular choice of Hilbert space $\cH$.

\begin{lem}
  \label{lem:extension}
 Let $\cH$ be a unitarily invariant space on $\bB_d$ such that the polynomials
 are multipliers of $\cH$, let $\cK$ be a Hilbert space and
  let $\Phi: A(\cH) \to \cB(\cK)$ be a bounded
  linear map. Then the following are equivalent:
  \begin{enumerate}[label=\normalfont{(\roman*)}]
    \item 
  $\Phi$ extends to a weak-$*$ continuous
  linear map from $\Mult(\cH)$ to $\cB(\cK)$.
  \item Whenever $(p_n)$ is a sequence of polynomials such that
  $||p_n||_{\Mult(\cH)} \le 1$ for all $n \in \bN$ and $\lim_{n \to \infty} p_n(z) = 0$ for all $z \in \bB_d$,
  the sequence $(\Phi(p_n))$ converges to $0$ in the weak-$*$ topology of $\cB(\cK)$.
  \end{enumerate}
If $\Phi$ is multiplicative, then so is its weak-$*$ continuous
  extension to $\Mult(\cH)$. Moreover, the set
  of $\Mult(\cH)$-Henkin functionals forms a norm closed subspace of $A(\cH)^*$.
\end{lem}

\begin{proof}
  The implication (i) $\Rightarrow$ (ii) is trivial. For the proof of the reverse implication,
  we first observe that unitary invariance of $\cH$ implies that
  there exists a strongly continuous one-parameter unitary group $(U_t)_{t \in \bR}$ on $\cH$
  defined by $(U_t f)(z) = f(e^{i t} z)$ for $f \in \cH$ and $z \in \bB_d$.
  If $\varphi \in \Mult(\cH)$, then $U_t M_\varphi U_t^* = M_{\varphi^{(t)}}$,
  where $\varphi^{(t)}(z) = \varphi(e^{i t} z)$.
  Consequently, if $\varphi \in \Mult(\cH)$, then $\varphi^{(t)}$
  belongs to $\Mult(\cH)$ with $||\varphi^{(t)}||_{\Mult(\cH)} = ||\varphi||_{\Mult(\cH)}$,
  and the map $t \mapsto \varphi^{(t)}$ is continuous in the strong operator topology.
  Using these two facts, a routine argument involving
  the Fej\'er kernel (cf. \cite[Lemma I 2.5]{Katznelson04})
  shows that if $\varphi \in \Mult(\cH)$
  and if $\varphi = \sum_{n=0}^\infty \varphi_n$ is the homogeneous expansion of
  $\varphi$ as an analytic function on $\bB_d$, then the sequence of polynomials
  $(\psi_n)$ defined by
  \begin{equation*}
    \psi_n = \frac{1}{n+1} \sum_{k=0}^n \sum_{j=0}^k \varphi_j
  \end{equation*}
  satisfies $||\psi_n||_{\Mult(\cH)} \le ||\varphi||_{\Mult(\cH)}$ and converges
  to $\varphi$ in the strong operator topology. With this observation in hand,
  the proof of \cite[Lemma 1.1]{Eschmeier97} carries over to the present setting.
  Moreover,
  the first additional claim follows from the fact that multiplication
  is separately continuous in the weak-$*$ topology.

  To show that the set of
  $\Mult(\cH)$-Henkin functionals is norm closed inside of $A(\cH)^*$, observe that
  the argument in the first part of the proof shows that the unit ball of $A(\cH)$
  is weak-$*$ dense in the unit ball of $\Mult(\cH)$. Thus, the contraction
  \begin{equation*}
    \Mult(\cH)^* \to A(\cH)^*, \quad \rho \mapsto \rho \big|_{A(\cH)},
  \end{equation*}
  restricts to an isometry between the predual of $\Mult(\cH)$ and the set of $\Mult(\cH)$-Henkin
  functionals. Since the predual of $\Mult(\cH)$ is complete, it follows
  that the set of $\Mult(\cH)$-Henkin functionals is closed in $A(\cH)^*$. (Alternatively,
  this can also be seen by using the characterization (ii) of $\Mult(\cH)$-Henkin functionals.)
\end{proof}

\begin{rem}
  Since the multiplier norm dominates the supremum norm over $\ol{\bB_d}$,
  it also follows from Lemma \ref{lem:extension} that every classical
  Henkin measure is a $\Mult(\cH)$-Henkin measure.
  The converse is false in general. Indeed, there are $\Mult(H^2_2)$-Henkin measures
  that are not classical Henkin measures \cite{Hartz17}.
\end{rem}

According to Henkin's theorem \cite{Henkin68}, see also \cite[Theorem 9.3.1]{Rudin08}, classical
Henkin measures form a band. Clou\^atre and Davidson showed that $\Mult(H^2_d)$-Henkin
measures form a band as well \cite[Theorem 5.4]{CD16}. Their proof
generalizes to $\Mult(\cH)$-Henkin measures.
If $\nu,\mu \in M(\partial \bB_d)$ are complex measures, then
we write $\nu \ll \mu$ if $\nu \ll |\mu|$.

\begin{lem}
  \label{lem:band}
  Let $\cH$ be a unitarily invariant space on $\bB_d$ such that the polynomials
  are multipliers of $\cH$. Then the $\Mult(\cH)$-Henkin measures form a band. That is,
  if $\mu,\nu \in M(\partial \bB_d)$ such that $\mu$ is $\Mult(\cH)$-Henkin and $\nu \ll \mu$, then $\nu$ is $\Mult(\cH)$-Henkin.
\end{lem}

\begin{proof}
While the relevant part of the proof of \cite[Theorem 5.4]{CD16} does generalize to this setting, 
  other parts do not easily generalize. Thus, for clarity, we include the proof here.
  Given a subset $E$ of $\partial \mathbb{B}_d$, let
  $\| f \|_{E}$ denote the value $\sup_{z \in E} |f(z)|.$

  By the Glicksberg-K\"onig-Seever decomposition theorem and Henkin's theorem (see Theorem 9.4.4, Theorem 9.3.1
  and Section 9.8 in \cite{Rudin08}),
  we can decompose $\mu = \mu_a + \mu_s$,
  where $\mu_a$ is $H^\infty(\bB_d)$-Henkin and $\mu_s$ is totally singular.
  Since every $H^\infty(\bB_d)$-Henkin measure is $\Mult(\cH)$-Henkin, the assumption implies
  that $\mu_s = \mu - \mu_a$ is $\Mult(\cH)$-Henkin.
  Since
  $\nu \ll \mu$, by the Radon-Nikodym theorem, there is a function $h \in L^1(|\mu|)$
  such that $\nu = h \mu = h \mu_a + h \mu_s$. By Henkin's theorem \cite[Theorem 9.3.1]{Rudin08},
  $h \mu_a$ is $H^\infty(\bB_d)$-Henkin and hence, $\Mult(\cH)$-Henkin. It therefore
  suffices to show that $h \mu_s$ is $\Mult(\cH)$-Henkin.

  Let $\varepsilon > 0$ and let $g$ be a continuous function on $\partial \bB_d$
  such that $||g - h||_{L^1(|\mu_s|)} < \varepsilon$.
  Since $\mu_s$ is totally singular, Rainwater's lemma (see \cite[Lemma 9.4.3]{Rudin08})
  shows that $\mu_s$ is concentrated on an $F_\sigma$ set which is totally null.
  Thus, there exists a totally null compact set
  $K \subset \partial \bB_d$ 
  such that $||\mu_s - \mu_s \big|_K||_{M(\partial \bB_d)} < \varepsilon / (||g||_{\partial \bB_d} + 1)$.
  By Bishop's theorem (see \cite[Theorem 10.1.2]{Rudin08}), there exists a function in the ball algebra
  which agrees with $g$ on $K$ and is bounded by $||g||_{\partial \bB_d}$.
 Thus, there is a polynomial $p$ such that $||p - g||_{K} < \varepsilon$
  and $||p||_{\partial \bB_d} \le ||g||_{\partial \bB_d}$.
  Clearly, $p \mu_s$ is $\Mult(\cH)$-Henkin since $\mu_s$ is $\Mult(\cH)$-Henkin. Moreover,
  \begin{align*}
    ||p \mu_s - h \mu_s||_{A(\cH)^*}
    &\le ||p \mu_s - h \mu_s||_{M(\partial \bB_d)} \\
    &\le ||(p - g) \mu_s |_K ||_{M(\partial \bB_d)} + ||(p - g) (\mu_s - \mu_s|_K)||_{M(\partial \bB_d)} \\
& \qquad    + ||g \mu_s - h \mu_s||_{M(\partial \bB_d)} \\
    &\le |\mu_s|(K) \, || p - g||_K + ||p - g||_{\partial \bB_d} ||\mu_s - \mu_s|_K||_{M(\partial \bB_d)} \\
& \qquad    + || g- h ||_{L^1(|\mu_s|)} \\
    &\le \varepsilon(3 + ||\mu_s||_{M(\partial \bB_d)} ). 
  \end{align*}
  Since $\varepsilon > 0$ was arbitrary and since the set of $\Mult(\cH)$-Henkin functionals
  is norm closed in $A(\cH)^*$ by Lemma \ref{lem:extension}, it follows that integration against $h \mu_s$ is $\Mult(\cH)$-Henkin, as desired.
\end{proof}

The spectral measure of a commuting tuple
of normal operators (see Subsection \ref{ss:normal})
is only unique up to mutual absolute continuity. Nevertheless, Lemma \ref{lem:band}
shows that it is meaningful to say that the spectral measure of a commuting
spherical unitary is $\Mult(\cH)$-Henkin.
Using Lemma \ref{lem:extension} and Lemma \ref{lem:band}, the following result can now be proved
just like \cite[Lemma 3.1]{CD16a}.

\begin{lem}
  \label{lem:unitary_abs_cont}
  Let $\cH$ be a unitarily invariant space on $\bB_d$ such that the polynomials are multipliers of $\cH$,
  and  let $U$ be a spherical unitary on a separable Hilbert space. Then $U$ is $\Mult(\cH)$-absolutely continuous if and only if the spectral
  measure of $U$ is $\Mult(\cH)$-Henkin. \qed
\end{lem}

We also obtain an answer to (Q2) in terms of minimal dilations of an operator tuple.
Recall from the discussion at the end of Subsection \ref{ss:dilations} that
minimal dilations are not necessarily unique.
The following result generalizes \cite[Theorem 3.2]{CD16a}, at least in the separable case.
Since we are dealing with dilations as opposed to co-extensions, the proof is somewhat more
complicated.

\begin{prop}
  \label{prop:ac_dilation}
  Let $\cH$ be a regular unitarily invariant space on $\bB_d$. Let $T = (T_1,\ldots,T_d)$
  be a tuple of commuting operators on a separable Hilbert space $\cK$ and let $M_z^{\kappa} \oplus U
  \in \cB(\cH^\kappa \oplus \cL)$ be a minimal dilation of $T$, where $U$ is spherical unitary.
  Then $T$ is $\Mult(\cH)$-absolutely
  continuous if and only if the spectral measure of $U$ is $\Mult(\cH)$-Henkin.
\end{prop}

\begin{proof}
  If the spectral measure of $U$ is $\Mult(\cH)$-Henkin, then $U$ is $\Mult(\cH)$-absolutely
  continuous by Lemma \ref{lem:unitary_abs_cont}. Since $M_z^{\kappa}$ is $\Mult(\cH)$-absolutely
  continuous and since $p(T) = P_{\cK} p(M_z^\kappa \oplus U) \big|_{\cK}$ for every polynomial $p$,
  it follows that $T$ is $\Mult(\cH)$-absolutely continuous as well.

  Conversely, suppose that $T$ is $\Mult(\cH)$-absolutely continuous.
  Let $E$ denote the projection-valued spectral measure of $U$ and let
  \begin{equation*}
    M = \{(x,y) \in \cL \oplus \cL: \langle E(\cdot)x,y \rangle \text{ is $\Mult(\cH)$-Henkin} \}.
  \end{equation*}
  We have to show that $M = \cL \oplus \cL$.
  For $y \in \cL$, let
  \begin{equation*}
    M_y = \{x \in \cL: (x,y) \in M \}.
  \end{equation*}
  Since the map 
  \begin{equation*}
    \cL \to A(\cH)^*, \quad x \mapsto \langle E(\cdot) x,y \rangle,
  \end{equation*}
  is linear and continuous, and since the set of all $\Mult(\cH)$-Henkin functionals forms
  a norm closed subspace of $A(\cH)^*$ by Lemma \ref{lem:extension}, we see that each $M_y$
  is a closed subspace of $\cL$.
  
  We claim that each $M_y$ is reducing for $U$. To this end,
  note that
  \begin{equation*}
    \int_{\partial \bB_d} p \, d \langle E(\cdot) U_i x,y \rangle
    = \langle p(U) U_i x, y \rangle  = \int_{\partial \bB_d} p z_i d \langle E(\cdot)x,y \rangle
  \end{equation*}
  and
  \begin{equation*}
    \int_{\partial \bB_d} p \, d \langle E(\cdot) U_i^* x,y \rangle
    = \langle p(U) U_i^* x, y \rangle  = \int_{\partial \bB_d} p \ol{z_i} d \langle E(\cdot)x,y \rangle
  \end{equation*}
  holds for every polynomial $p$. Clearly, the measures $z_i \langle E(\cdot)x, y  \rangle$
  and $ \ol{z_i} \langle E(\cdot)x,y \rangle$ are absolutely continuous with respect to
  $\langle E(\cdot)x,y  \rangle$, hence Lemma \ref{lem:band} shows that if $x \in M_y$, then also
  $U_i x \in M_y$ and $U_i^* x \in M_y$. Thus, each $M_y$ is reducing for $U$.

  Next, we show that if $x,y \in \cK$, then $(P_{\cL} x, P_{\cL} y) \in M$. Let
  $p$ be a polynomial. Then
  \begin{align*}
    \int_{\partial \bB_d} p \, d \langle E(\cdot) P_{\cL} x, P_{\cL} y \rangle =
    \langle p(U) P_{\cL} x, P_{\cL} y \rangle
    &= \langle p(M_z^{\kappa} \oplus U) x, y \rangle - \langle p(M_z^{\kappa}) P_{\cH^{\kappa}} x, P_{\cH^{\kappa}} y \rangle  \\
    &= \langle p(T) x, y \rangle - \langle p(M_z^{\kappa}) P_{\cH^{\kappa}} x, P_{\cH^{\kappa}} y \rangle.
  \end{align*}
  Since $T$ and $M_z^{\kappa}$ are both $\Mult(\cH)$-absolutely continuous, we see that
  $\langle E(\cdot) P_{\cL} x, P_{\cL} y \rangle$ is $\Mult(\cH)$-Henkin, so that
  $(P_{\cL} x, P_{\cL} y) \in M$.

  Finally, since $M_z^{\kappa} \oplus U$ is a minimal dilation of $T$, it follows in particular
  that $\cL$ is the smallest reducing subspace for $U$ that contains $P_{\cL} \cK$. From the preceding
  two paragraphs, we therefore deduce that if $y \in P_{\cL} \cK$, then $M_y = \cL$.
  Another application of Lemma \ref{lem:band} shows that if $(x,y) \in M$, then also $(y,x) \in M$, since
  $\langle E(\cdot) x , y \rangle = \ol{\langle E(\cdot)y,x \rangle}$. Thus, for every $y \in \cL$,
  the space $M_y$ contains $P_{\cL} \cK$. Using minimiality of the dilation again,
  we conclude that $M_y = \cL$ for every $y \in \cL$, so that $M = \cL \oplus \cL$, as desired.
\end{proof}

\section{\texorpdfstring{$\Mult(\cH)$-absolute continuity}{Mult(H)-absolute continuity}} \label{sec:ac}

The goal of this section is to prove  our main result, Theorem \ref{thm:cnu_abs_con}. 
Before proceeding with the proof, we require a lemma about totally singular measures
on $\partial \bB_d$.
It is a well-known consequence of the classical F. and M. Riesz theorem that if $\mu$ is a measure
on the unit circle which is singular with respect to Lebesgue measure, then the analytic polynomials
are weak-$*$ dense in $L^\infty(\mu)$ (see, for example, \cite[Exercise VI.7.10]{Conway91}).
The following lemma is a generalization of this fact
to higher dimensions.

\begin{lem}
  \label{lem:reductive}
  Let $\mu$ be a positive totally singular measure on $\partial \bB_d$. Then the analytic polynomials
  are weak-$*$ dense in $L^\infty(\mu)$.
\end{lem}

\begin{proof}
  Let $f \in L^1(\mu)$ such that
  \begin{equation*}
    \int_{\partial \bB_d} p f \, d \mu = 0
  \end{equation*}
  for all analytic polynomials $p$. By the Hahn-Banach theorem, we need to show that $f = 0 \in L^1(\mu)$.
  To this end, observe that the measure $f \mu$ annihilates the ball algebra, hence
  the Cole-Range theorem (see \cite[Theorem 9.6.1]{Rudin08}) shows that there exists a
  representing measure $\rho$ for the origin such that $f \mu \ll \rho$. On the other hand, since $\mu$
  is totally singular, $\mu$ and hence $f \mu$ is singular with respect to $\rho$. Thus, $f \mu = 0$,
  as desired.
\end{proof}

Let $\cH$ be a regular unitarily invariant space on $\bB_d$.
To  prove Theorem \ref{thm:cnu_abs_con}, we will use the dilation of
Theorem \ref{thm:H-dilation} and so, require information about 
operator tuples of the form $M_z^{\kappa} \oplus U$,
where $M_z = (M_{z_1},\ldots,M_{z_d}) \in \cB(\cH)^d$ and $U \in \cB(\cL)^d$ is a spherical unitary.
Let us assume that $\cL$ is separable, and let $\mu \in M(\partial \bB_d)$
be a scalar valued spectral measure of $U$.
By the Glicksberg-K\"onig-Seever
decomposition theorem \cite[Theorem 9.4.4]{Rudin08} and Henkin's theorem \cite[Theorem 9.3.1]{Rudin08}
(see also the discussion in \cite[Section 9.8]{Rudin08}),
we can decompose $\mu = \mu_a + \mu_s$, where $\mu_a$ is $H^\infty(\bB_d)$-Henkin
and $\mu_s$ is totally singular. Correspondingly, $U = U_a \oplus U_s$,
where $\mu_a$ is a spectral measure of $U_a$ and $\mu_s$ is a spectral measure
of $U_s$.
Let $\cL_a,\cL_s$ denote the spaces on which $U_a,U_s$ act, respectively.

The following lemma plays a key role in the proof of Theorem \ref{thm:cnu_abs_con}.

\begin{lem}
  \label{lem:key_lemma}
  In the setting of the preceding paragraph, let $\cM$ be the unital weak-$*$ closed subalgebra of
  $\cB(\cH^{\kappa} \oplus \cL_a \oplus \cL_s)$ generated by
  $M_z^{\kappa} \oplus U_a \oplus U_s$, and let
  \begin{equation*}
    \cM_s = \{N \in W^*(U_s): 0 \oplus 0 \oplus N \in \cM \}.
  \end{equation*}
  Then $\cM_s$ is a weak-$*$ closed ideal of the von Neumann algebra $W^*(U_s)$.
\end{lem}

\begin{proof}
  It is clear that $\cM_s$ is a weak-$*$ closed subspace of $W^*(U_s)$. To show that it
  is an ideal, let $N \in \cM_s$. For $p \in \bC[z_1,\ldots,z_d]$,
  \begin{equation*}
    0 \oplus 0 \oplus p(U_s) N =
    p (M_z^{\kappa} \oplus U_a \oplus U_s)
  (0 \oplus 0 \oplus N)
  \end{equation*}
  belongs to $\cM$ and hence, $p(U_s) N \in \cM_s$.
  The spectral measure $\mu_s$ of $U_s$
  is totally singular and so, Lemma \ref{lem:reductive} implies that the analytic polynomials are weak-$*$
  dense in $L^\infty(\mu_s)$.
  Recall that $W^*(U_s)$ and $L^\infty(\mu_s)$
  are isomorphic as von Neumann algebras via an isomorphism that sends the tuple $U_s$
  to $(z_1,\ldots,z_d)$. Thus, operators of the form $p(U_s)$, where $p$ is an analytic polynomial,
  are weak-$*$ dense in $W^*(U_s)$.
  Since $\cM_s$ is weak-$*$ closed, we conclude 
  that $R N \in \cM_s$ for all $R \in W^*(U_s)$, and by the commutativity of $W^*(U_s)$,
$N R$ is also in $\cM_s$. Thus, $\cM_s$ is an ideal.
\end{proof}

We are now in a position to prove Theorem \ref{thm:cnu_abs_con}.

\begin{proof}[Proof of Theorem \ref{thm:cnu_abs_con}]
  Let $T=(T_1, \dots, T_d) \in \cB(\cK)^d$ be a commuting
  operator tuple on a Hilbert space $\cK$ that admits an $A(\cH)$-functional calculus.
  Suppose that $T$
  is not $\Mult(\cH)$-absolutely continuous. We will show that $T$ has a spherical unitary summand.

  In a first step, we observe that it suffices to consider the case where $\cK$ is separable.
  Indeed, since $T$ is not $\Mult(\cH)$-absolutely continuous and since the weak-$*$
  and weak operator topologies coincide on bounded subsets of $\cB(\cK)$,
  there exist by Lemma \ref{lem:extension} a sequence
  of polynomials $(p_n)$ in the unit ball of $\Mult(\cH)$ that converges to $0$ in the weak-$*$ topology of $\Mult(\cH)$
   and $x,y \in \cK$ such that the sequence
  $(\langle p_n(T) x,y \rangle)$ does not converge to $0$. Let $\cK_0$ be
  the smallest reducing subspace for $T$ that contains $x$ and $y$.
  Then $\cK_0$ is separable, and $T \big|_{\cK_0}$
  is not $\Mult(\cH)$-absolutely continuous. Moreover, if $T \big|_{\cK_0}$ has a spherical unitary
  summand, then so does $T$. Thus, we may assume that $\cK$ is separable.

  By Theorem
  \ref{thm:H-dilation}, the tuple $T$ dilates to $M_z^{\kappa} \oplus U$, where
  $\kappa$ is a countable cardinal and $U$ is a spherical unitary on a separable Hilbert space.
  As in the discussion preceding Lemma \ref{lem:key_lemma}, we can decompose $U = U_a \oplus U_s$.
  Since $T$ is not $\Mult(\cH)$-absolutely continuous, Lemma \ref{lem:extension} guarantees the 
  existence of a sequence of polynomials $(p_n)$ in the unit ball of $\Mult(\cH)$
  such that $(p_n)$ converges to   $0$
  in the weak-$*$ topology of $\Mult(\cH)$ but the bounded sequence $(p_n(T))$ does not converge
  to $0$ in the weak-$*$ topology of $\cB(\cK)$.
 Since $\cK$ is separable, the weak-$*$ topology
  on the unit ball of $\cB(\cK)$ is metrizable.
  Then, by passing to a subsequence, we can assume 
  that $(p_n(T))$ actually converges to a non-zero operator $A \in \cB(\cK)$.

  Observe that $(p_n(M_z^\kappa))$ converges to $0$ weak-$*$.
  Moreover, $(p_n)$ is bounded in the supremum norm on $\bB_d$ and converges
  to $0$ pointwise in $\bB_d$. Hence, as
  the spectral measure of $U_a$ is $H^{\infty}(\bB_d)$-Henkin, the sequence
  $(p_n(U_a))$ converges to $0$ weak-$*$ as well. From
  \begin{equation*}
    p_n(T) = P_{\cK} p_n( M_z^{\kappa} \oplus U_a \oplus U_s) \big|_{\cK}
  \end{equation*}
  and the fact that $(p_n(T))$ converges to the non-zero operator $A$,
  we deduce that
  the bounded sequence $(p_n(U_s))$ does not converge to $0$ weak-$*$. By passing to another subsequence,
  we can assume that $(p_n(U_s))$ converges to a non-zero contraction $R \in W^*(U_s)$,
  so that $A = P_{\cK}( 0 \oplus 0 \oplus R) \big|_{\cK}$.
  
  Let $\cM$ and $\cM_s$ be defined as in Lemma \ref{lem:key_lemma}. By our previous arguments, 
  $R$ belongs to $\cM_s$. By Lemma \ref{lem:key_lemma}, 
  $\cM_s$ is a weak-$*$ closed ideal of the von Neumann algebra $W^*(U_s).$
  Thus, there exists an orthogonal projection $P \in \cM_s$
  such that $P S = S$ for all $S \in \cM_s$ (see, for example, \cite[Section III.1.1.13]{Blackadar06}).
  Since $M^{\kappa}_z \oplus U_a \oplus U_s$ dilates $T$, the space $\cK$ is semi-invariant for $M^{\kappa}_z \oplus U_a \oplus U_s$
  and hence for every operator in $\cM$. Therefore, the map
  \begin{equation*}
    \cM \to \cB(\cK), \quad S \mapsto P_{\cK} S \big|_{\cK},
  \end{equation*}
  is multiplicative, which we will use repeatedly.
  Since $\cK$ is semi-invariant and thus reducing
  for the self-adjoint operator $0 \oplus 0 \oplus P \in \cM$,
  we see that $Q = P_{\cK} 0 \oplus 0 \oplus P \big|_{\cK}$
  is an orthogonal projection. Recall that $\cM$ is commutative and so, $Q$ and $T_k$ commute
  for $1 \le k \le d$.
  Moreover, from $P R = R$, we deduce that $Q A = A$, which implies $Q \neq 0$.

  We finish the proof by showing that the restriction of  $T$ to the range of $Q$
  is a spherical unitary. Applying Lemma \ref{lem:key_lemma} again, we see that $P (U_s)_k$
  and $P (U_s)_k^*$ belong to $\cM_s$ for $1 \le k \le d$,
  hence $\cK$ is reducing for each operator $0 \oplus 0 \oplus P (U_s)_k$.
  Moreover,
  \begin{equation*}
    Q T_k = P_{\cK}( 0 \oplus 0 \oplus P (U_s)_k) \big|_{\cK}
  \end{equation*}
  for $1 \le k \le d$. In particular, each $Q T_k$ is the compression of a normal operator
  to a reducing subspace and thus, is normal. Moreover,
  \begin{equation*}
    \sum_{k=1}^d Q T_k (Q T_k)^*
    = P_{\cK} ( 0 \oplus 0 \oplus P ) \big|_{\cK} = Q,
  \end{equation*}
  which completes the proof.
\end{proof}

Moreover, in the presence of a spherical unitary summand, we obtain the following characterization of $\Mult(\cH)$-absolute continuity.
\begin{thm}
  \label{thm:general_functional_calculus}
  Let $\cH$ be a regular unitarily invariant space on $\bB_d$, and let $T = (T_1,\ldots,T_d)$
  be a commuting tuple of operators on a separable Hilbert space that admits an $A(\cH)$-functional calculus.
Let $T = T_{cnu} \oplus U$ be the decomposition of $T$ from Proposition \ref{prop:cnu_decom}.
  Then $T$ is $\Mult(\cH)$-absolutely continuous
  if and only if the spectral measure of $U$ is $\Mult(\cH)$-Henkin.
\end{thm}

\begin{proof}
  This follows immediately from Theorem \ref{thm:cnu_abs_con} and Lemma \ref{lem:unitary_abs_cont}.
\end{proof}

\section{Complete Nevanlinna-Pick Kernels} \label{sec:cnp}

In this section, we obtain a refined version of Theorem \ref{thm:cnu_abs_con} when $\mathcal{H}$ is a complete
Nevanlinna-Pick space.
Specifically, we use the fact that
in this case, there exists a more checkable condition for $T$
to admit an $A(\cH)$-functional calculus.
Background material on Nevanlinna-Pick spaces can be found in the book \cite{AM02}.

Let $\cH$ be a regular unitarily invariant space on $\bB_d$ with reproducing kernel
$K(z,w) = \sum_{n=0}^\infty a_n \langle z,w \rangle^n$. A straightforward
generalization of
\cite[Theorem 7.33]{AM02} 
shows that $\cH$ is an irreducible complete Nevanlinna-Pick space if and only if the sequence
$(b_n)$ given by
\begin{equation}
  \label{eqn:a_b}
  \sum_{n=1}^\infty b_n t^n = 1 - \frac{1}{\sum_{n=0}^\infty a_n t^n}
\end{equation}
satisfies $b_n \ge 0$ for all $n \ge 1$.

For example, the class of regular unitarily invariant complete Nevanlinna-Pick spaces on $\bB_d$
includes the spaces with reproducing kernels
\begin{equation*}
  K(z,w) = \frac{1}{(1- \langle z,w \rangle)^\alpha}
\end{equation*}
for $\alpha \in (0,1]$ and the spaces with kernels
\begin{equation*}
  K(z,w) = \sum_{n=0}^\infty (n+1)^s \langle z,w \rangle^n
\end{equation*}
for $s \le 0$, and hence in particular the Drury-Arveson space and the Dirichlet space (see Subsection
\ref{ss:rkhs}).

Let $\cH$ be a regular unitarily invariant complete Nevanlinna-Pick space on $\bB_d$ with kernel $K$
and let $(b_n)$ be the sequence of Equation \eqref{eqn:a_b}.
We write $1/K(T,T^*) \ge 0$ if
\begin{equation*}
  \sum_{n=1}^N b_n \sum_{|\alpha| = n} \binom{n}{\alpha} T^\alpha (T^*)^\alpha \le I
\end{equation*}
for all $N \in \bN$. This definition goes back to work of Agler \cite{Agler82}. For more
discussion, see \cite[Section 5]{CH16}. The following result is \cite[Theorem 5.4]{CH16}.

\begin{thm}
  \label{thm:H-coextension}
   Let $\cH$ be a regular unitarily invariant complete Nevanlinna-Pick space on $\bB_d$
   with kernel $K$.
   Let $T = (T_1,\ldots,T_d)$ be a tuple of commuting operators on a Hilbert space. Then the
  following are equivalent:
  \begin{enumerate}[label=\normalfont{(\roman*)}]
    \item The tuple $T$ satisfies $1/K(T,T^*) \ge 0$.
    \item The tuple $T$ admits an $A(\cH)$-functional calculus.
  \end{enumerate}
\end{thm}

We remark that if $\cH$ is not a complete Nevanlinna-Pick space, then one can often still
make sense of the condition $1/K(T,T^*) \ge 0$, see \cite{AM00a,AEM02}. In general, however, not every
tuple which admits an $A(\cH)$-functional calculus satisfies $1/K(T,T^*) \ge 0$,
as the example of the Bergman space and the unilateral shift shows. Rather, the condition $1/K(T,T^*) \ge 0$
is related to the existence of co-extensions, whereas admitting an $A(\cH)$-functional calculus
is equivalent to the existence of dilations (see Theorem \ref{thm:H-dilation}). In the complete
Nevanlinna-Pick setting, there is no difference, see \cite{CH16} for more discussion.

The following refinement of Theorem \ref{thm:general_functional_calculus} in the complete Nevanlinna-Pick
setting is almost immediate. 
 
 \begin{cor}
  \label{cor:NP_functional_calculus}
  Let $\cH$ be a regular unitarily invariant complete Nevanlinna-Pick space on $\bB_d$ with kernel $K$.
  Let $T = (T_1,\ldots,T_d)$
  be a commuting tuple of operators on a separable Hilbert space and let
  $T = T_{cnu} \oplus U$ be the decomposition of Proposition \ref{prop:cnu_decom}. Then the following
  are equivalent:
  \begin{enumerate}[label=\normalfont{(\roman*)}]
    \item The tuple $T$ admits an $A(\cH)$-functional calculus and is $\Mult(\cH)$-absolutely continuous.
    \item The tuple $T$ satisfies $1/K(T,T^*) \ge 0$ and the spectral
      measure of $U$ is $\Mult(\cH)$-Henkin.
  \end{enumerate}
\end{cor}

\begin{proof}
  This follows from Theorem \ref{thm:H-coextension} and Theorem \ref{thm:general_functional_calculus}.
\end{proof}

\bibliographystyle{amsplain}
\bibliography{literature}

\end{document}